\documentclass[9pt,paper]{article}
\usepackage[all]{xy}
\usepackage{latexsym,amsthm,amssymb,amscd,amsmath}
\usepackage{amsfonts}
\newtheorem{thm}{Theorem}[section]

\newtheorem{prop}[thm]{Proposition}

\newtheorem{lem}[thm]{Lemma}
\newtheorem{Def}[thm]{Definition}

\newcounter{alphthm}
\newcommand{\be}{\be}
\newcommand{\ee}{\end{equation}}
\newcommand{\ben}{\begin{enumerate}}
\newcommand{\een}{\end{enumerate}}
\newcommand{\pa}{{\partial}}

\title{On the \textrm{k}-nullity foliations in Finsler geometry and completeness}
\author{B. Bidabad\footnote{\tiny{Faculty of Mathematics, Amirkabir University of Technology, Tehran, Iran.(email:bidabad@aut.ac.ir)}}\ \ \ \ \ and\ \ \ \ \ M. Rafie-Rad\footnote{\tiny{Faculty of Mathematics, Mazandaran University, Babolsar, Iran.(email: m.rafiei.rad@gmail.com)}}}

\date{}
\begin{document}
\maketitle
\begin{abstract}
Here, a Finsler manifold $(M,F)$ is considered with corresponding
curvature tensor,  regarded as $2$-forms on the bundle of non-zero tangent vectors. Certain
subspaces of the tangent spaces of $M$ determined by the curvature are introduced and called $k$-nullity foliations of the curvature operator.
 It is shown that if the dimension of foliation is constant then the distribution is involutive and each maximal integral manifold is totally geodesic. Characterization of the $k$-nullity foliation is given, as well as some results concerning constancy of the flag curvature, and
completeness of their integral manifolds, providing completeness of $(M,F)$. The introduced $k$-nullity space is a natural extension of nullity space in Riemannian geometry, introduced by S. S. Chern and N. H. Kuiper and enlarged to Finsler setting by H. Akbar-Zadeh and contains it as a special case.
\end{abstract}

\section{Introduction}
Foliations of manifolds occur naturally in various geometric contexts. They arise in connections with some essential topics as vector fields without singularities, integrable
$m$-dimensional distributions, submersions and fibrations, actions of Lie groups, direct constructions
of foliations such as Hopf fibrations, Reeb foliations and finally they appear in the  existence study of solution of certain differential equations. In the later case, S. Tanno in \cite{T} applied the concept of the {\bf k}-nullity spaces to achieve a complete proof for the famous Obata Theorem which is a subject of numerous rigidity results in Riemannian geometry.
The nullity space of the Riemannian curvature tensor was first studied by S. S. Chern and N. H. Kuiper \cite{CK} in 1952. They have shown that, if the index of nullity, $\mu$, of a Riemannian manifold is locally constant, then the manifold admits a locally integrable $\mu$-dimensional distribution whose integral submanifolds are locally flat. O. Kowalski and M. Sekizawa have proved that vanishing of the index of nullity in some senses, results that the tangent sphere bundle is a space of negative scalar curvature \cite{KoSe}.

The concept of nullity spaces are generalized to the ${\bf k}$-nullity spaces in Riemannian geometry in a number of works such as \cite{CM,Gr} and \cite{Mz}.

In this work we answer to the following natural questions: Is there any extension for the concept of {\bf k}-nullity space in Finsler geometry? Is its maximal integral manifold totally geodesic? And finally is its maximal integral manifold complete, provided that $(M,F)$ is complete? Fortunately, the answer to these questions is affirmative. More precisely, we obtain the following results.
\begin{thm}
\label{mainthm1}
Let $(M,F)$ be a Finsler manifold for which the index of {\bf k}-nullity $\mu_{\bf k}$ be constant on an open subset $U\subseteq M$. Then, the local ${\bf k}$-nullity distribution on $U$ is completely integrable.
\end{thm}

\begin{thm}
\label{coincidence}
The ${\bf k}$-nullity space of a Finsler manifold $(M,F)$ at a point $x\in M$, coincides with the kernel of the related curvature operator of $\Omega$.
\end{thm}
D. Ferus has proved that the maximal integral manifolds of nullity foliation are totally geodesic \cite{F2}. This result has been extended to the Finsler case by H. Akbar-Zadeh \cite{A5}. Here, we prove the same result for {\bf k}-nullity foliation in Finsler manifolds.
\begin{thm}
\label{lema}
Let $(M,F)$ be a Finsler manifold. If the ${\bf k}$-nullity space is locally constant on the open subset $U$ of $M$, then every $\bf
k$- nullity integral manifold $N$ in $U$ is an auto-parallel Finsler submanifold with
non-negative constant flag curvature {\bf k}. Moreover, $(N,\tilde{F})$ is a $P$-symmetric space.
\end{thm}
The completeness of the nullity foliations is studied by D. Ferus in \cite{F1}. The similar result is carried out for Finsler manifolds by H. Akbar-Zadeh \cite{A5} in 1972.
\begin{thm}
\label{mainthm2}
Let $(M,F)$ be a complete Finsler manifold and $G$ an open subset of $M$ on which $\mu_{\bf k}$ is minimum. Then, every integral manifold of the {\bf k}-nullity foliation in $G$ is a complete submanifold of $M$.
\end{thm}
It is worth mentioning that M. Sekizawa and S. Tachibana have studied $k^{th}$ nullity foliations as another generalization of Chern and Kuiper's nullity in Riemannian geometry by considering $k^{th}$ consecutive derivative of the curvature tensor \cite{Se,TS}.
\section{Preliminaries and terminologies.}
\subsection{Regular connections and Finsler manifolds.}
 Let $M$ be a
connected differentiable manifold of dimension $n$. We adopt here the notations and terminologies of \cite{A1}. Denote the bundle of tangent vectors of $M$ by
$p:TM\longrightarrow M$, the fiber bundle of non-zero tangent
vectors of $M$ by
$\pi:TM_0\longrightarrow M$ and the
pulled-back tangent bundle by $\pi^*TM\longrightarrow TM_0$. Any point of $TM_0$ is denoted
by $z=(x,v)$, where $x=\pi z\in M$ and $v\in T_{\pi z}M$. By $TTM_0$
we denote the tangent bundle of $TM_0$ and by $\varrho$ the canonical
linear mapping
$$\varrho:TTM_0\longrightarrow \pi^*TM,$$ where $ \varrho=\pi_*$. For all $z\in TM_0$, let ${\cal V}_zTM$ be the
set of vertical vectors at $z$, that is, the set of vectors which
are tangent to the fiber through $z$. Equivalently, ${\cal V}_zTM=\ker \pi_*$
where
$\pi_*:TTM_0\longrightarrow TM$ is the linear tangent mapping.\\
Let $\nabla$ be a linear connection on the vector bundle
$\pi^*TM\longrightarrow TM_0$. We define a linear mapping
$$\mu:TTM_0\longrightarrow \pi^*TM,$$ by
$\mu(\hat{X})=\nabla_{\hat{X}}{\bf v}$ where $\hat{X}\in TTM_0$
and ${\bf v}$ is the canonical section of $\pi^*TM$.

The connection $\nabla$ is said to be {\it regular}, if $\mu$ defines an isomorphism between ${\cal V}TM_0$ and
$\pi^*TM$. In this case, there is the horizontal distribution ${\cal H}TM$ such that we have the Whitney sum:
\[
TTM_0={\cal H}TM\oplus{\cal V}TM.
\]
This decomposition permits to write a vector $\hat{X}\in TTM_0$ into the form $\hat{X}=H\hat{X}+V\hat{X}$ uniquely. In the sequel we denote all vector fields on $TM_0$ by $\hat{X}, \hat{Y}$, etc and the corresponding sections of $\pi^*TM$ by $X=\varrho(X)$, $Y=\varrho(Y)$, etc, respectively, unless otherwise specified.

The structural equations of the regular connection $\nabla$ are given by:
 \begin{equation}
\tau(\hat{X},\hat{Y})=\nabla_{\hat{X}}Y-\nabla_{\hat{Y}}X-\varrho[\hat{X},\hat{Y}],
\end{equation}
\begin{equation}
\Omega(\hat{X},\hat{Y})Z=\nabla_{\hat{X}}\nabla_{\hat{Y}}Z-\nabla_{\hat{Y}}\nabla_{\hat{X}}Z
-\nabla_{[\hat{X},\hat{Y}]}Z,
\end{equation}
where $X=\varrho(\hat{X})$, $Y=\varrho(\hat{Y})$,
$Z=\varrho(\hat{Z})$ and $\hat{X}$, $\hat{Y}$ and $\hat{Y}$ are vector fields on $TM_0$. The tensors $\tau$ and $\Omega$ are called
{\it Torsion} and {\it Curvature} tensors of $\nabla$, respectively. They
determine two torsion tensors denoted here by $S$ and $T$ and three
curvature tensors denoted by $R$, $P$ and $Q$ defined by:
\[
S(X,Y)=\tau(H\hat{X},H\hat{Y}),\ \ \
T(\dot{X},Y)=\tau(V\hat{X},H\hat{Y}),
\]\[
R(X,Y)=\Omega(H\hat{X},H\hat{Y}),\ \
P(X,\dot{Y})=\Omega(H\hat{X},V\hat{Y}),\ \
Q(\dot{X},\dot{Y})=\Omega(V\hat{X},V\hat{Y}),
\]
where $X=\varrho(\hat{X})$,\ $Y=\varrho(\hat{Y})$,\
$\dot{X}=\mu(\hat{X})$ and $\dot{Y}=\mu(\hat{Y})$. The tensors $R$, $P$ and $Q$ are called $hh-$, $hv-$ and $vv-$curvature tensors, respectively. Using the
Jacobi identity for three vector fields ${\hat X}$, ${\hat Y}$ and
${\hat Z}$, one obtains {\it the Bianchi identities} for a regular
connection $\nabla$ with curvature 2-forms $\Omega$, as follows:
\begin{equation}
\label{C1}
\sigma\Omega({\hat X},{\hat Y})Z=\sigma\nabla_{{\hat Z}}\tau({\hat
X},{\hat Y})+\sigma\tau({\hat Z},[{\hat X},{\hat Y}]),
\end{equation}
\begin{equation}
\label{C2}
\sigma\nabla_{{\hat Z}}\Omega({\hat X},{\hat Y})+\sigma\Omega({\hat
Z},[{\hat X},{\hat Y}])=0,
\end{equation}
where, $\sigma$ denotes the circular permutation in the set $\{{\hat
X}, {\hat Y},{\hat Z}\}$.

Let $(x^i)$ be a local
chart with the domain $U\subseteq M$ and $(x^i,v^i)$ the induced
local coordinates on $\pi^{-1}(U)$ where
${\bf v}=v^i\frac{\partial}{\partial x^i}\in T_{\pi z}M$, where $i$ run over the range $1,2,...,n$. A {\it Finsler
structure} $F$ is defined to be a function $F$ on
$TM_0$ satisfying  the following conditions: \textbf{(1)}$ \  F>0\ \textrm{and}\ C^\infty\ \textrm{on}\ TM_0$, \textbf{(2)}$ \ F(x,\lambda v)=\lambda F(x,v),$ for every $\lambda>0$ and \textbf{(3)}$\ \ g_{ij}(x,v)=\frac{1}{2}\frac{\partial^2 F^2}{\partial
v^i\partial v^j}$ is positive definite. The pair $(M,F)$ is called a {\it Finsler manifold}.\\
 There is a unique regular
connection associated to $F$ such that:
 \begin{eqnarray}
 \nabla_{\hat{Z}}g&=&0,\nonumber\\
  S(X,Y)&=&0,\nonumber\\
  g(\tau(V\hat{X},\hat{Y}),Z)&=&g(\tau(V\hat{X},\hat{Z}),Y),\nonumber
  \end{eqnarray}
where $X=\varrho(\hat{X})$, $Y=\varrho(\hat{Y})$ and
$Z=\varrho(\hat{Z})$ for all $\hat{X}$, $ \hat{Y}$,
$\hat{Z}\in TTM_0$. The regular connection $\nabla$ is called the {\it Cartan connection}.
Given an induced natural coordinates on $\pi^{-1}(U)$, the coefficients of $\nabla$ can be written as follows:
\[
\nabla_{\partial_j}\partial_i=\Gamma^k_{\ ij}\partial_k,\ \
\nabla_{\overset{\bullet}{\partial}_j}\partial_i=C^k_{\ ij}\partial_k,
\]
where, $\partial_i=\frac{\partial}{\partial x^i},\
\overset{\bullet}{\partial}_i=\frac{\partial}{\partial v^i}$ and $\Gamma^k_{\ ij}$ and $C^k_{\ ij}$ are smooth functions defined on $\pi^{-1}(U)$. One can observe that components of the second torsion
tensor $T$ coincides with components of Cartan tensor $C$ in this coordinates, that is $T_{ijk}=\frac{1}{2}\overset{\bullet}{\partial}_kg_{ij}$, where, $T_{ijk}=g_{ir}T^r_{\ jk}$.
It can be shown that the set $\{\delta_j\}$ defined by $\delta_j=\partial_j-\Gamma^k_{\ 0j}\overset{\bullet}{\partial}_k$ form a local frame field for the horizontal space ${\cal H}TM$. Assume that $\nabla_{\delta_j}\partial_i=\overset{*}{\Gamma}{^k_{\ ij}}\partial_k$.
One can easily see that $\overset{*}{\Gamma}{^k_{\ ij}}$ is symmetric with respect to the
indices $i$ and $j$. The curvature operator $\Omega(\hat{X},\hat{Y})$ of Cartan connection is anti-symmetric in the following sense
\begin{equation}
\label{antisymmetric}
g(\Omega(\hat{X},\hat{Y})Z,W)=-g(\Omega(\hat{X},\hat{Y})W,Z),
\end{equation}
where $\hat{X},\hat{Y}\in{\cal X}(TM_0)$, $Z=\varrho(\hat{Z})$ and $W=\varrho(\hat{W})$.
The $hv-$curvature tensor $P$ and the $vv-$ curvatures tensor $Q$ of the Cartan connection $\nabla$ are
given respectively by
\begin{equation}
\label{P formula}
P^i_{\ jkl}=\nabla^iT_{jkl}-\nabla_jT^i_{\ kl}+T^i_{\
kr}\nabla_0T^r_{\ jl}-T^r_{\ kj}\nabla_0T^i_{\ rl},
\end{equation}
\begin{equation}
\label{Q formula}
Q^i_{\ jkl}=T^i_{\ rl}T^r_{\ jk}-T^i_{\ rk}T^r_{\ jl}.
\end{equation}

Among the Finsler manifolds, there are some classes determined by non-Riemannian quantities. One of them which is appeared in the present work is the $P$-symmetric Finsler manifolds required a kind of partial symmetry in the indices of $P$. This class of Finsler manifolds, was introduced by M. Matsumoto an H. Shimada in \cite{MS} and \cite{M1}, and has been extensively studied by many authors.

The curvature tensor $P^i_{\ jkl}$ can be decomposed into the sum of two symmetric and anti-symmetric tensors with respect to the indices $k$ and $l$, that is to say $P={^sP}+{^aP}$. By means of Eq.(\ref{P formula}) the symmetric tensor ${^sP}$ can be written in the following form:
\begin{equation}
\label{Ps formula}
{^sP}^i_{\ jkl}=\nabla^iT_{jkl}
+\frac{1}{2}\{T^i_{\
kr}\nabla_0T^r_{\ jl}-T^r_{\ kj}\nabla_0T^i_{\ rl}+T^i_{\
lr}\nabla_0T^r_{\ jk}-T^r_{\ lj}\nabla_0T^i_{\ rk}\}.
\end{equation}
A Finsler manifold is said to be {\it P-symmetric} if $P(X,Y)=P(Y,X),$ $\forall X,Y\in\Gamma(\pi^*TM)$.
$P$-symmetric spaces are closely related to the Finsler manifolds of isotropic sectional curvature.
 In this relation the following result is well-known.
\begin{prop} \cite{MS}
A Finsler manifold is $P$-symmetric if and only if $\nabla_{\hat{\bf v}}Q=0$.
\end{prop}
Next we consider the
{\it Berwald connection} $D$ which
is not metric-compatible but a torsion free regular connection relative to $F$. There is the following
relation between the connections $\nabla$ and $D$
\begin{equation}
\label{BerCarrel}
D_{H\hat{X}}Y=\nabla_{H\hat{X}}Y+(\nabla_{\hat{{\bf v}}}T)(X,Y),\ \ \ \ D_{V\hat{X}}Y=(V\hat{X}.Y^i)\partial_i\,
\end{equation}
where the vector field $\hat{{\bf v}}=v^i\delta_i$ is the canonical geodesic spray of
$F$. If we assume
$D_{\delta_j}\partial_i=G^k_{\ ij}\partial_k,\ \
$
then Eqs.(\ref{BerCarrel}) can be written in the following local form:
\[
G^i_{\ jk}=\overset{*}{\Gamma}{^i_{\ jk}}+\nabla_0T^i_{\ jk},\ \ \ D_{\overset{\bullet}{\partial}_j}\partial_i=0.
\]
It is clear from Eq.(\ref{BerCarrel}) that the
connections $D$ and $\nabla$ associate to the same geodesic spray, since we have $\nabla_{\hat{X}}{\bf v}=D_{\hat{X}}{\bf v}$. The metric tensor $g$ related to the Finsler structure $F$ is parallel along any geodesic of Berwald connection, that is equivalent to $D_{\hat{\bf v}}g=0$.
The Berwald connection $D$ admits the $hh-$ curvature tensors $H$ and the $hv-$ curvature tensors $G$  with the components $H^i_{\
jkl}$ and $G^i_{\ jkl}$.
$G^i_{\ jkl}$ and $H^i_{\ jkl}$ can be determined by
\[
G^i_{\ jkl}=\overset{\bullet}{\partial}_l\overset{\bullet}{\partial}_k\overset{\bullet}{\partial}_jG^i=\overset{\bullet}{\partial}_lG^i_{\ jk},
\]\[
H^i_{\ jkl}=(\delta_kG^i_{\ jl}-G^i_{\ ljs}G^s_{\ k})-(\delta_lG^i_{\ jk}-G^i_{\ kjs}G^s_{\ l})
+G^i_{\ rk}G^i_{\ jl}-G^i_{\ rl}G^i_{\ jk}.
\]
Let $z\in TM_0$ and ${\cal P}({\bf v},X)\subseteq T_{\pi z}M$ be a plane,
generated by ${\bf v}$ and a linearly independent vector $X$ in
$T_{\pi z}M$. The {\it flag curvature} at the point $z\in TM_0$ with
respect to ${\cal P}({\bf v},X)$ is denoted by ${\bf K}(z,{\cal P}({\bf v},X))$
and is defined as follows:
\[
{\bf K}(z,{\cal P}({\bf v},X))=\frac{g(R(X,{\bf v}){\bf
v},X)}{g(X,X)F^2-g(X,{\bf v})^2},
\]
where, $R$ denotes the $hh-$curvature of Cartan connection \cite{S}.
Note that, the flag curvature ${\bf K}(z,{
\cal P}({\bf v},X))$ does not depend on the choice of
Berwald and Cartan connection, since, after a simple
calculation. In fact, one can easily show that
\begin{equation}
\label{HR}
H(X,{\bf v}){\bf v}=R(X,{\bf v}){\bf v}.
\end{equation}
The Finsler manifold $(M,F)$ is said to be of {\it scalar flag curvature at the point} $z\in TM_0$ if ${\bf
K}(z,P({\bf v},X))$ does not depend on the choice of the plane ${\cal P}({\bf v},X)$
and it is said to be of {\it scalar flag curvature} if it is
of scalar flag curvature at all points $z\in TM_0$. In this case we have:
\[
R(X,{\bf v}){\bf v}={\bf K}(z)\{F^2X-g(X,{\bf v}){\bf v}\},\ \ \ \forall X\in\Gamma(\pi^*TM).
\]
\subsection{Finsler submanifolds.}
Let $S$ be a k-dimensional embedded submanifold of the Finsler manifold $(M,F)$ defined by
embedding ${\bf i}:S\longrightarrow M$. We identify a point $\tilde{x}\in S$ and a tangent vector $\widetilde{X}\in T_{\tilde{x}}S$ by
by ${\bf i}(\tilde{x})$ and ${\bf i}_*\widetilde{X}$, respectively. Hence, $T_{\tilde{x}}S$ can be considered as a subspace
of $T_{\tilde{x}}M$. The embedding ${\bf i}$ induces a map $\tilde{{\bf i}}={\bf i}_*:TS_0\longrightarrow TM_0$. If we
identify a point $\widetilde{z}\in TS_0$ with its image $\tilde{{\bf i}}(\widetilde{z})$, then $TS_0$ can be considered as a
sub-fiber bundle of $TM_0$. Restricting the map $\pi:TM_0\longrightarrow M$ to $TS_0$, we obtain the mapping $q:TS_0\longrightarrow M$. Denote by $\hat{T}S={\bf i}^*TM$, the pulled back bundle of $TM$.
The Finsler metric $g$ on $M$ induces a Finsler metric on $S$ which is denoted  by $\widetilde{g}$. Given any point
$\widetilde{x}=q(\widetilde{z})\in S$, where $\widetilde{z}\in TS_0$, we denote by $N_{q(\widetilde{z})}$ the orthogonal complementary subspace of $T_{q(\widetilde{z})}M$
in $\hat{T}_{q(\widetilde{z})}S$. Therefore we have the Whitney sum
\begin{equation}
\label{decompo}
\hat{T}_{q(\tilde{z})}S=T_{q(\tilde{z})}S\oplus N_{q(\tilde{z})}.
\end{equation}
The above decomposition defines the two projection maps ${\bf P}_1$ and ${\bf P}_2$ as follows:
\[
{\bf P}_1:\hat{T}S\longrightarrow TS,
\]\[
{\bf P}_2:\hat{T}S\longrightarrow N,
\]
where $N=\bigcup_{\tilde{z}\in TS_0} N_{q(\tilde{z})}$.
We have $q^*\hat{T}S=q^*TS\oplus N$. $N$ is called {\it the normal fiber bundle}. We denote by $\rho$ the canonical linear mapping $TTS_0\longrightarrow q^*TS$, that is,
 $\rho=q_*$. Let $\widetilde{X}$ and $\widetilde{Y}$ be two vector fields on $TS_0$. Given $\tilde{z}\in TS_0$,
 $(\nabla_{\widetilde{X}}Y)_{\tilde{z}}$ belongs to $\hat{T}_{q(\tilde{z})}S$. Therefore, using the decomposition (\ref{decompo}), we get
\begin{equation}
\label{second fundamental form}
\nabla_{\widetilde{X}}Y=\widetilde{\nabla}_{\widetilde{X}}Y+\alpha(\widetilde{X},Y),
\end{equation}
where, $\nabla$ is the Cartan connection, $Y=\rho(\widetilde{Y})$, $\widetilde{\nabla}_{\widetilde{X}}Y\in T_{q(\tilde{z})}S$ and $\alpha(\widetilde{X},Y)\in N_{q(\tilde{z})}$. $\alpha$ is called {\it the second fundamental form of $S$}.
From Eq.(\ref{second fundamental form}), it follows that, $\widetilde{\nabla}$ is a covariant derivative in the
vector bundle $q^*TS\longrightarrow TS_0$ and satisfies $\widetilde{\nabla}\widetilde{g}=0$. $\widetilde{\nabla}$ is called the {\it tangential covariant derivation}. $\alpha(\widetilde{X},\rho(\widetilde{Y}))$
is a bilinear form possessing its values in $N$. Let us denote by $\widetilde{\tau}$ the torsion
tensor of $\widetilde{\nabla}$. Then, we have
\[
{\bf P}_1\tau(\widetilde{X},\widetilde{Y})=\widetilde{\tau}(\widetilde{X},\widetilde{Y})
=\widetilde{\nabla}_{\widetilde{X}}Y-\widetilde{\nabla}_{\widetilde{Y}}X-\rho[\widetilde{X},\widetilde{Y}],
\]\[
{\bf P}_2\tau(\widetilde{X},\widetilde{Y})=\alpha(\widetilde{X},Y)-\alpha(\widetilde{Y},X),
\]
where $X=\rho(\widetilde{X})$ and $Y=\rho(\widetilde{Y})$.
The submanifold $S$ is said to be {\it totally geodesic} at a point $\tilde{x}\in S$ if, for every tangent vector $\widetilde{X}\in T_{\tilde{x}}S$, the geodesic $\gamma(t)$ of $M$ in the direction of $\widetilde{X}$ lies in $S$ for small values of the parameter $t$. If $S$ is totally geodesic at every point of $S$, it is called a {\it totally geodesic submanifold} of $M$.
\setcounter{alphthm}{0}
\begin{thm}
\cite{A1} Let $S$ be a submanifold of the Finsler manifold $(M,F)$ with the second fundamental form $\alpha$. Then, $S$ is a totally geodesic submanifold if and only if $\alpha(\widetilde{X},{\bf v})=0$, for all $\widetilde{X}\in{\cal X}(TS_0)$.
\end{thm}
The submanifold $S$ is also said to be an auto-parallel submanifold of $M$ if the second fundamental form $\alpha$
vanishes identically. Note that, in the Riemannian manifolds, the concepts of auto-parallel and totally geodesic submanifolds coincide. Clearly, every auto-parallel submanifold is also totally geodesic. Notice that, on an auto-parallel submanifold $S$, the induced connection $\widetilde{\nabla}$ coincides with the Cartan connection of the induced Finsler structure $\widetilde{F}=\tilde{{\bf i}}^*F$.\\

\subsection{Nullity space of curvature operator in Finsler geometry.}
Let $(M,F)$ be a Finsler manifold and $\nabla$ the Cartan connection related to $F$.
 Given any point $z\in TM_0$, consider the subspace of ${\cal H}_zTM$ defined by:
\[
N_z:=\{\hat{X}\in {\cal H}_zTM|\ \ \Omega(\hat{X},\hat{Y})=0,\ \ \forall
\hat{Y}\in {\cal H}_zTM\},
\]
where, $\Omega$ is the curvature operator of $\nabla$. For any point
$z\in TM_0$ where $\pi z=x$. The subspace ${\cal N}_x=\varrho(N_z)\subset
T_xM$ is linearly isomorphic
to $N_z$. ${\cal N}_x$ is called the {\it nullity space} of the curvature operator on the Finsler manifold $(M,F)$ at the point $x\in M$, while ${\cal N}$ will denote the {\it field of nullity spaces}. Its
orthogonal complementary space in $T_xM$ is called the {\it co-nullity space at $x$} and is denoted by
$\overset{\perp}{{\cal N}}_x$. Every
element of ${\cal N}_x$ is called a {\it nullity
vector}. The non-negative integer valued function $\mu_0:M\longrightarrow I\!{N}$ defined by $\mu_0(p)=\dim{\cal N}_p$ is called {\it the index of nullity} and $\mu_0(p)$ is called the index of nullity at the point $p\in M$. Nullity space is called {\it locally constant} if given
any $x\in M$, there is a neighborhood $U$ of $x$ such that the function
$\mu_0$ is constant on $U$. In this case, the correspondence $x\in
U\mapsto {\cal N}_x$ is a distribution called the
{\it nullity distribution} on $U$. In the sequel we assume $0<\mu_0<n$ unless otherwise specified.

Let $\ker_x \Omega$ be the kernel of the operator $\Omega$,
that is
\begin{equation}
\label{nullity}
\ker_x \Omega=\{Z\in T_xM|\ \Omega(\hat{X},\hat{Y})Z=0,\ \forall
\hat{X},\hat{Y}\in {\cal H}_zTM\}.
\end{equation}
Akbar-Zadeh proved that, ${\cal N}_x=\ker_x
\Omega$ and moreover, if the nullity space is locally constant on $U$, then the nullity distribution on $U$
is completely integrable. This is an extension of the similar result in Riemannian manifolds, established by Maltz \cite{Mz} and Gray \cite{Gr}.
Akbar-Zadeh proved the following result:
\begin{thm}
Let $(M,F)$ be a complete Finsler manifold and $G$ an open subset in $M$ on which $\mu_0$ is minimum. Then, every
nullity manifold is a geodesically complete submanifold of
$M$.
\end{thm}
\section{{\bf k}-Nullity space of Cartan connection's curvature operator.}
Let $(M,F)$ be an n-dimensional Finsler manifold endowed with the Cartan connection $\nabla$. The aim of this section is to associate to $(M,F)$ a {\bf k}-nullity space of the Cartan connection's curvature operator. We first introduce the concept of {\bf k}-nullity space as a natural extension of nullity space in Finsler geometry containing nullity space as a special case ${\bf k}=0$. Furthermore, we study fundamental properties of {\bf k}-nullity spaces. Given any non-negative real number {\bf k}, we define the tensors $\eta^{\bf k}$ and $\bar{\Omega}$ as follows
\[
\eta^{\bf k}(\hat{X},\hat{Y})Z={\bf k}\{g(Y,Z)X-g(X,Z)Y\}+{^aP}(X,\dot{Y})Z,
\]
\begin{equation}
\label{def bar omega}
\bar{\Omega}(\hat{X},\hat{Y})Z=\Omega(\hat{X},\hat{Y})Z-\eta^{\bf
k}(\hat{X},\hat{Y})Z,
\end{equation}
where, $\hat{X},\hat{Y},\hat{Z}\in {\cal X}(TM_0)$, $X=\varrho(\hat{X})$, $Y=\varrho(\hat{Y})$, $Z=\varrho(\hat{Z})$ and ${^aP}$ is the anti-symmetric part of $hv-$curvature tensor $P(X,Y)$.
We refer to $\bar{\Omega}$ as the related curvature operator of $\Omega$. The local representation of $\bar{\Omega}(H\hat{X},H\hat{Y})$ is given by $$\bar{\Omega}^i_{\ jkl}=R^i_{\
jkl}-{\bf k}\{g_{jk}\delta^i_{\ l}-g_{jl}\delta^i_{\ k}\},$$ and we have from Eq.(\ref{def bar omega})
\begin{equation}
\label{S0}
\bar{\Omega}(H\hat{X},V\hat{Y})={^sP}(X,\dot{Y}),
\end{equation}
where, $\dot{Y}=\mu(\hat{Y})$.
Notice that, Eq.(\ref{Ps formula}) yields
\begin{equation}
\label{S}
\bar{\Omega}(H\hat{X},V\hat{Y}){\bf v}={^sP}(X,\dot{Y}){\bf v}=0,
\end{equation}
where ${\bf v}$ is the canonical section of $\pi^*TM$ given by ${\bf v}=v^i\pa_i$. Given any point $z\in TM_0$, we define the subspace $N^{\bf k}_z$ of ${\cal H}_zTM$ by
\[
N^{{\bf k}}_z:=\{\hat{X}\in {\cal H}_zTM|\ \bar{\Omega}(\hat{X},\hat{Y})=0,\ \ \ \forall
\hat{Y}\in {\cal H}_zTM\}.
\]
For any point
$z\in TM_0$ and $\pi z=x$, we consider the subspace ${\cal N}^{{\bf k}}_x=\varrho(N^{{\bf k}}_z)\subset
T_xM$. Clearly, the subspace ${\cal N}^{\bf k}_x=\varrho(N^{\bf k}_z)\subset
T_xM$ is linearly isomorphic to $N^{\bf k}_z$, since $\varrho$ is a linear isomorphism between ${\cal H}TM$ and $\pi^*TM$.

Now, we are in position to define a non-Riemannian {\bf k}-nullity space on Finsler manifolds.
\setcounter{alphthm}{0}
\begin{Def}
Let $(M,F)$ be a Finsler manifold. ${\cal N}^{\bf k}_x$ is called the {\bf k}-{\it nullity space} of the curvature operator on the Finsler manifold $(M,F)$ at the point $x\in M$, while ${\cal N}^{\bf k}$ will denote the {\it field of {\bf k}-nullity spaces}. Its
orthogonal complementary space in $T_xM$ is denoted by
${\scriptstyle \overset{\perp}{{\cal N}}{^{\bf k}_x}}$. Every
element of ${\cal N}^{\bf k}_x$ is called a {\it {\bf k}-nullity
vector}. The non-negative integer valued function $\mu_{\bf k}:M\longrightarrow I\!{N}$ defined by $\mu_{\bf k}(p)=\dim{\cal N}^{\bf k}_p$ is called the index of {\bf k}-nullity at the point $p\in M$. {\bf k}-nullity space is called {\it locally constant} if given
any $x\in M$, there is a neighborhood $U$ of $x$ such that the function
$\mu_{\bf k}$ is constant on $U$. In this case, the correspondence $x\in
U\mapsto {\cal N}^{\bf k}_x$ is a distribution called the
{\it {\bf k}-nullity distribution} on $U$.
\end{Def}
The function $\mu_{\bf k}:M\longrightarrow I\!\!{N}$ is upper semi-continuous. In the sequel we assume that $0<\mu_{\bf k}<n$ unless otherwise specified.\\
Observe that, the following relations hold for $\eta^{\bf k}$:
\begin{equation}
\label{etar1}
\sigma\eta^{\bf k}({\hat X},{\hat Y})Z=0,\ \ \  \nabla_{\hat{Z}}\eta^{\bf k}=0,\ \ \ \forall\hat{X},\hat{Y},\hat{Z}\in {\cal H}TM.
\end{equation}
where, $\sigma$ is a circular permutation on the set $\{\hat{X},\hat{Y},\hat{Z}\}$. Thus, it is clear that we have:
\begin{equation}
\label{etar2}
\sigma\bar{\Omega}({\hat X},{\hat Y})Z=\sigma\Omega({\hat X},{\hat Y})Z,\ \ \ \forall\hat{X},\hat{Y},\hat{Z}\in {\cal H}TM.
\end{equation}
The tensor $\bar{\Omega}$ has somehow the same algebraic properties as $\Omega$. The following properties of $\bar{\Omega}$ are easily verified:
\begin{lem}
\label{lemi}
The following statements hold for $\bar{\Omega}$:\\
$(1)\ \  \sigma\bar{\Omega}({\hat X},{\hat Y})Z=\sigma\nabla_{{\hat
Z}}\tau({\hat X},{\hat Y})+\sigma\tau({\hat Z},[{\hat X},{\hat
Y}]),$\\
$(2)\ \  \sigma\nabla_{{\hat Z}}\bar{\Omega}({\hat X},{\hat
Y})+\sigma\bar{\Omega}({\hat Z},[{\hat X},{\hat Y}])=0,$\\
$(3)\ \  g(\bar{\Omega}(\hat{X},\hat{Y})Z,W)=-g(\bar{\Omega}(\hat{X},\hat{Y})W,Z)$,
where, $\hat{X},\hat{Y},\hat{Z},\hat{W}\in {\cal H}TM$ and $\sigma$ is a
circular permutation in the set $\{\hat{X},\hat{Y},\hat{Z}\}$.
\end{lem}
 \begin{proof} The proof is a simple application of Bianchi identities, Eq.(\ref{antisymmetric}), Eq.(\ref{etar1}) and Eq.(\ref{etar2}).
 \end{proof}
\textit{Proof of Theorem \ref{mainthm1}}.
Let $\hat{X}$, $\hat{Y}$ and $\hat{Z}$ be three horizontal vector fields on $TM_0$ such that $\hat{X},\hat{Y}\in N^{\bf k}_z$. Taking into account Eq.(\ref{S}) and Eq.(\ref{HR}), by a straightforward computation, we have $$\varrho[\hat{X},\hat{Y}]=[X,Y]_\pi,$$
\begin{equation}
\label{mu}
\mu([\hat{X},\hat{Y}])=-\Omega(\hat{X},\hat{Y}){\bf v}=-\eta^{\bf k}(\hat{X},\hat{Y}){\bf v},
\end{equation}
\begin{equation}
\label{K0}
H[\hat{X},\hat{Y}]=[\hat{X},\hat{Y}]+\eta^{\bf k}(\hat{X},\hat{Y})v^r\overset{\bullet}{\partial}_r.
\end{equation}
In this case, the relation (2) in Lemma \ref{lemi} reduces to
\[
\bar{\Omega}(\hat{X},[\hat{Y},\hat{Z}])+\bar{\Omega}(\hat{Y},[\hat{Z},\hat{X}])
+\bar{\Omega}(\hat{Z},[\hat{X},\hat{Y}])=0
\]
The last equation can be written in the following form:
\begin{equation}
\label{K1}
\bar{\Omega}(\hat{X},V[\hat{Y},\hat{Z}])+\bar{\Omega}(\hat{Y},V[\hat{Z},\hat{X}])
+\bar{\Omega}(\hat{Z},[\hat{X},\hat{Y}])=0.
\end{equation}
Following Eq.(\ref{S0}) and Eq.(\ref{mu}), first and second terms of Eq.(\ref{K1}) become:
\begin{equation}
\label{K2}
\bar{\Omega}(\hat{X},V[\hat{Y},\hat{Z}])={^sP}(X,\mu[\hat{Y},\hat{Z}])=-{^sP}(X,\eta^{\bf k}(\hat{Y},\hat{Z}){\bf v})
\end{equation}
\[
={\bf k}g(Z,{\bf v}){^sP}(X,Y)-{\bf k}g(Y,{\bf v}){^sP}(X,Z),
\]
\begin{equation}
\label{K3}
\bar{\Omega}(\hat{Y},V[\hat{Z},\hat{X}])={^sP}(Y,\mu[\hat{Z},\hat{X}])=-{^sP}(Y,\eta^{\bf k}(\hat{Z},\hat{X}){\bf v})
\end{equation}
\[
={\bf k}g(X,{\bf v}){^sP}(Y,Z)-{\bf k}g(Z,{\bf v}){^sP}(Y,X).
\]
By means of Eq.(\ref{K2}) and Eq.(\ref{K3}) and the symmetry property ${^sP}(X,Y)={^sP}(Y,X)$, Eq.(\ref{K1}) can be written in the following form:
\[
\bar{\Omega}(\hat{Z},[\hat{X},\hat{Y}]+\eta^{\bf k}(\hat{X},\hat{Y})v^r\overset{\bullet}{\partial}_r)=0,
\]
Following Eq.(\ref{K0}), the last equation becomes:
\[
\bar{\Omega}(\hat{Z},H[\hat{X},\hat{Y}])=0,\ \ \ \hat{Z}\in{\cal H}_zTM.
\]
Indeed $H[\hat{X},\hat{Y}]\in N^{\bf k}_z$ and $[X,Y]=\varrho[\hat{X},\hat{Y}]=\varrho(H[\hat{X},\hat{Y}])\in {\cal N}^{\bf k}_x$. Therefore, ${\bf k}$-nullity distribution is involutive or completely integrable.\qed
\\
 Considering the kernel of the operator
$\bar{\Omega}$
\[
\ker_x \bar{\Omega}=\{Z\in T_xM|\ \bar{\Omega}(\hat{X},\hat{Y})Z=0,\ \
\hat{X},\hat{Y}\in {\cal H}_zTM\},
\]
we shall show that  ${\cal N}^{{\bf k}}_x=\ker_x
\bar{\Omega}$.\\
\textit{Proof of Theorem \ref{coincidence}}.
 Let $\hat{X}$, $\hat{Y}$ and $\hat{Z}$ be three horizontal vector fields on $TM_0$ such that $\hat{X},\hat{Y}\notin N^{\bf k}_z$ but $\hat{Z}\in N^{\bf k}_z$. In this case, the relation (1) in Lemma \ref{lemi} reduces to
\[
\bar{\Omega}(\hat{X},\hat{Y})Z=\tau(\hat{X},[\hat{Y},\hat{Z}])+\tau(\hat{Y},[\hat{Z},\hat{X}])+\tau(\hat{Z},[\hat{X},\hat{Y}])
\]
On the other hand, for every vector field $\hat{W}\in{\cal X}(TM_0)$, we have:
\begin{equation}
\label{I}
g(\bar{\Omega}(\hat{X},\hat{Y})Z,W)=
g(\tau(\hat{X},[\hat{Y},\hat{Z}]),W)+g(\tau(\hat{Y},[\hat{Z},\hat{X}]),W)+g(\tau(\hat{Z},[\hat{X},\hat{Y}]),W).
\end{equation}
Considering Eq.(\ref{mu}), we have the following relations for the torsion tensor $\tau$.
\begin{equation}
\label{To1}
\tau(\hat{X},[\hat{Y},\hat{Z}])=T(\hat{X},\mu[\hat{Y},\hat{Z}])={\bf k}g(Y,{\bf v})T(X,Z)-{\bf k}g(Z,{\bf v})T(X,Y),
\end{equation}
\begin{equation}
\label{To2}
\tau(\hat{Y},[\hat{Z},\hat{X}])=T(\hat{Y},\mu[\hat{Z},\hat{X}])={\bf k}g(Z,{\bf v})T(Y,X)-{\bf k}g(X,{\bf v})T(Y,Z),
\end{equation}
\begin{equation}
\label{To3}
\tau(\hat{Z},[\hat{X},\hat{Y}])=T(\hat{Z},\mu[\hat{X},\hat{Y}])={\bf k}g(X,{\bf v})T(Z,Y)-{\bf k}g(Y,{\bf v})T(Z,X).
\end{equation}
Replacing Eqs.(\ref{To1}),(\ref{To3}) and (\ref{To3}), in Eq.(\ref{I}) we obtain
\[
g(\bar{\Omega}(\hat{X},\hat{Y})Z,W)=2{\bf k}\{g(Y,{\bf v})g(T(X,Z),W)+g(Z,{\bf v})g(T(Y,X),W)
\]\[
+g(X,{\bf v})g(T(Z,Y),W)\}.
\]
As a consequence of the relation (3) in Lemma \ref{lemi}, the left hand side of the previous equation is anti-symmetric with respect to $W$ and $Z$. Thus, it follows that $$2{\bf k}\{g(Y,{\bf v})g(T(X,Z),W)+g(X,{\bf v})g(T(Z,Y),W)\}=0.$$
Since $W$ is arbitrarily chosen, we have the following relation:
\begin{equation}
\label{help}
g(Y,{\bf v})T(X,Z)+g(X,{\bf v})T(Z,Y)=0.
\end{equation}
From Eq.(\ref{mu}) one can conclude that
\[
\tau(\hat{Z},[\hat{X},\hat{Y}])=\tau(\hat{Z},V[\hat{X},\hat{Y}])=T(Z,\mu[\hat{X},\hat{Y}])
\]\[
={\bf k}g(X,{\bf v})T(Z,Y)-{\bf k}g(Y,{\bf v})T(Z,X).
\]
By anti-symmetry property of the tensor $T$ and Eq.(\ref{help}), we get
\[
\tau(\hat{Z},V[\hat{X},\hat{Y}])=
{\bf k}g(X,{\bf v})T(Z,Y)+{\bf k}g(Y,{\bf v})T(X,Z)=0.
\]
Plugging Eqs.(\ref{To1}),(\ref{To2}) and (\ref{To3}) into Eq.(\ref{I}) results
\[
g(\bar{\Omega}(\hat{X},\hat{Y})Z,W)=
g(\tau(\hat{X},[\hat{Y},\hat{Z}]),W)+g(\tau(\hat{Y},[\hat{Z},\hat{X}]),W)+g(\tau(\hat{Z},[\hat{X},\hat{Y}]),W)=0.
\]
Therefore, we have
\[
g(\bar{\Omega}(\hat{X},\hat{Y})Z,W)=g(\tau(\hat{Z},[\hat{X},\hat{Y}]),W)=0.
\]
Finally, since $W$ is arbitrarily chosen, we obtain the following equation,
\begin{equation}
\label{E1}
\bar{\Omega}(\hat{X},\hat{Y})Z=\tau(\hat{Z},[\hat{X},\hat{Y}])=T(Z,\mu[\hat{X},\hat{Y}])=0.
\end{equation}
The last equation shows that $Z\in \ker_x\bar{\Omega}$, that is ${\cal N}^{\bf k}_x\subseteq \ker_x\bar{\Omega}$ and $\ker\bar{\Omega}^\perp\subseteq \overset{\perp}{{\cal N}}{^{\bf k}_x}$. Now, let $W\in \overset{\perp}{{\cal N}}{^{\bf k}_x}$ and $U\in{\cal N}^{\bf k}_x$, we have
\[
g(\bar{\Omega}(\hat{X},\hat{Y})W,U)=-g(\bar{\Omega}(\hat{X},\hat{Y})U,W)=0.
\]
The previous equation shows that $\bar{\Omega}(\hat{X},\hat{Y})W\in\overset{\perp}{{\cal N}}{^{\bf k}_x}$, that is $Im_x\bar{\Omega}\subseteq\overset{\perp}{{\cal N}}{^{\bf k}_x}$. For every {\bf k}-nullity vector $U\in{\cal N}^{\bf k}_x$, Eq.(\ref{mu}) yields
\[
g(\mu([\hat{X},\hat{Y}])+\eta^{\bf k}(\hat{X},\hat{Y}){\bf v},U)=-g(\bar{\Omega}(\hat{X},\hat{Y}){\bf v},U)
=g(\bar{\Omega}(\hat{X},\hat{Y})U,{\bf v})=0.
\]
By definition of $\eta^{\bf k}$ and the fact that $X,Y\in\overset{\perp}{{\cal N}}{^{\bf k}_x}$, we obtain $g(\eta^{\bf k}(\hat{X},\hat{Y}){\bf v},U)=0$. Therefore,
\begin{equation}
\label{E2}
\mu([\hat{X},\hat{Y}])\in\overset{\perp}{{\cal N}}{^{\bf k}_x}.
\end{equation}
From which $g(\mu([\hat{X},\hat{Y}]),U)=0$.
Consider the following homomorphism of vector spaces:
\[
\Psi:\frac{T_xM}{\ker_x\bar{\Omega}}\cong Im_x\bar{\Omega}\longrightarrow\overset{\perp}{{\cal N}}{^{\bf k}_x},
\]
defined by $W+\ker_x\bar{\Omega}\mapsto \bar{\Omega}(\hat{X},\hat{Y})W$.
It is clear that $\Psi$ is one-to-one and thus it is onto and therefore, ${\scriptstyle \overset{\perp}{{\cal N}}{^{\bf k}_x}}=\ker_x{\scriptstyle \bar{\Omega}^\perp}$ and
${\cal N}^{\bf k}_x=\ker_x\bar{\Omega}$. This completes the proof of Theorem. \qed
\section{Auto-parallel {\bf k}-nullity maximal integral manifold}
\textit{Proof of Theorem \ref{lema}}. The method used here is inspired by Akbar-Zadeh's technic. Let $N$ be an integral manifold of {\bf k}-nullity distribution in $U$. For all vector fields $\widetilde{X},\widetilde{W}\in{\cal X}(TN_0)$ we have by means of Eq.(\ref{second fundamental form})
\begin{equation}
\label{second fundamental form2}
\nabla_{\widetilde{W}}X=\widetilde{\nabla}_{\widetilde{W}}X+\alpha(\widetilde{W},X),
\end{equation}
where, $\widetilde{\nabla}$ denotes the induced connection on $TN_0$, $X=\rho(\widetilde{X})$ and
$\alpha(\widetilde{W},X)$ is the second fundamental form of $N$.
\\
Let $\widetilde{X},\widetilde{Y}\in {\cal H}TN$ such that
$X,Y\in\overset{\perp}{{\cal N}}{^{\bf k}_x}$ and $U\in {\cal N}^{\bf k}_x$. By means of Theorem \ref{coincidence}, we have $\bar{\Omega}(\hat{X},\hat{Y})U=0$. Suppose that $\widetilde{Z}\in
N^{\bf k}_z$. It follows immediately from Eq.(\ref{second fundamental form2}) that the covariant derivative of $\bar{\Omega}$ along $\widetilde{Z}$ becomes
\begin{eqnarray*}
\nonumber
(\nabla_{\widetilde{Z}}\bar{\Omega}(\widetilde{X},\widetilde{Y}))U&=&\nabla_{\widetilde{Z}}\bar{\Omega}(\widetilde{X},\widetilde{Y})U-\bar{\Omega}(\widetilde{X},\widetilde{Y})\nabla_{\widetilde{Z}}U
\\&=&-\bar{\Omega}(\widetilde{X},\widetilde{Y})\nabla_{\widetilde{Z}}U\\&=&-\bar{\Omega}(\widetilde{X},\widetilde{Y})(\widetilde{\nabla}_{\widetilde{Z}}U+\alpha(\widetilde{Z},U))
\\&=&-\bar{\Omega}(\widetilde{X},\widetilde{Y})\alpha(\widetilde{Z},U).
\end{eqnarray*}
Therefore,
\[
(\nabla_{\widetilde{Z}}\bar{\Omega}(\widetilde{X},\widetilde{Y}))U+\bar{\Omega}(\widetilde{X},\widetilde{Y})\alpha(\widetilde{Z},U)=0.
\]
Using the identity (2) in Lemma \ref{lemi} and the above equation, we obtain
\begin{equation}
\label{E3}
\bar{\Omega}(\widetilde{X},\widetilde{Y})\alpha(\widetilde{Z},U)=\bar{\Omega}(\widetilde{Z},[\widetilde{X},\widetilde{Y}])U
={^sP}(\widetilde{Z},\mu[\widetilde{X},\widetilde{Y}])U.
\end{equation}
If we assume $\mu[\widetilde{X},\widetilde{Y}]=0$, then we have  $\bar{\Omega}(\widetilde{X},\widetilde{Y})\alpha(\widetilde{Z},U)=0$. On the other hand  $\alpha(\widetilde{Z},U)\in\overset{\perp}{{\cal N}}{^{\bf k}_x}$, it follows that $\alpha(\widetilde{Z},U)=0$. In this case, the integral manifold $N$ is an auto-parallel submanifold. Otherwise, assume that $\mu[\widetilde{X},\widetilde{Y}]\neq0$. Consider a basis $\{{\bf e}_1,{\bf e}_2,...,{\bf e}_n\}$ for $T_xM$ such that, the first $r$ vectors form a basis for ${\cal N}^{\bf k}_x$ and the rest $(n-r)$ vectors is a basis for $\overset{\perp}{{\cal N}}{^{\bf k}_x}$. In virtue of Eq.(\ref{E2}), without loss of generality, one can assume that the vector $\mu[\widetilde{X},\widetilde{Y}]$ is an element of basis $\{{\bf e}_{r-1},...,{\bf e}_{n}\}$. In the sequel, assume that the following indices run over the indicated ranges
\begin{equation}
\label{convention}
a,b=1,2,...,n,\ \ \ \ \alpha,\beta=1,2,...,r,\ \ \ \ i,j=r-1,...,n.
\end{equation}
Eq.(\ref{E1}) states that, in this basis, we have
\begin{equation}
\label{T Cartan}
T_{a\alpha j}=0
\end{equation}
Plugging it into Eq.(\ref{Ps formula}), it results
\[
{^sP}_{ia\beta j}=\nabla_iT_{a \beta j}
+\frac{1}{2}\{T_{i\beta r}\nabla_0T^r_{\ a j}-T^r_{\ \beta a}\nabla_0T_{irj}
+T_{ijr}\nabla_0T^r_{\ a \beta }-T^r_{\ ja}\nabla_0T_{ir\beta }\}=0.
\]
From the last equation, it results that ${^sP}(\widetilde{Z},\mu[\widetilde{X},\widetilde{Y}])U=0$. Eq.(\ref{E3}) implies that  $\bar{\Omega}(\widetilde{X},\widetilde{Y})\alpha(\widetilde{Z},U)=0$, that is to say $\alpha(\widetilde{Z},U)\in\ker\bar{\Omega}={\cal N}^{\bf k}_x$. It follows $\alpha(\widetilde{Z},U)=0$ and $N$ is an auto-parallel submanifold.

Denote the curvature 2-forms of $\widetilde{\nabla}$ by $\widetilde{\Omega}$. Since $N$ is an auto-parallel submanifold of $M$, its curvature tensors are given by
\[
\widetilde{\Omega}(H\widetilde{X},H\widetilde{Y})Z=\Omega(H\widetilde{X},H\widetilde{Y})Z={\bf k}\{\widetilde{g}(Y,Z)X-\widetilde{g}(X,Z)Y\},
\]\[
\widetilde{\Omega}(H\widetilde{X},V\widetilde{Y})Z=\Omega(H\widetilde{X},V\widetilde{Y})Z={^sP}(X,\dot{Y})Z,
\]\[
\widetilde{\Omega}(V\widetilde{X},V\widetilde{Y})Z=\Omega(V\widetilde{X},V\widetilde{Y})Z=Q(\dot{X},\dot{Y})Z,
\]
where, $\widetilde{X},\widetilde{Y}\in{\cal X}(TN_0)$ and $Z\in\Gamma(\pi^*TN)$. The above relation shows that, $N$ is a $P$-symmetric space. Indeed components of the $hh-$curvature $\widetilde{R}^\alpha_{\ \beta\gamma\theta}$ of $(N,\widetilde{F})$ are given by
\[
\widetilde{R}^\alpha_{\ \beta\gamma\theta}={\bf k}\{\widetilde{g}_{\beta\gamma}\delta^\alpha_{\ \theta}-\widetilde{g}_{\beta\theta}\delta^\alpha_{\ \gamma}\},
\]
where, $\widetilde{g}$ denotes the induced metric on $(N,\widetilde{F})$. Following Eq.(\ref{HR}), we have
\[
\widetilde{H}(X,{\bf v}){\bf v}=\widetilde{R}(X,{\bf v}){\bf v}={\bf k}\{\widetilde{g}({\bf v},{\bf v})Y-\widetilde{g}(Y,{\bf v}){\bf v}\},
\]
which shows that $N$ is of constant flag curvature ${\bf k}$. \qed
\subsection{Completeness of the {\bf k}-nullity foliation.}
\textit{Proof of Theorem \ref{mainthm2}}. Let $(M,F)$ be an $n$-dimensional Finsler manifold and $\gamma:[0,c)\longrightarrow N$ a geodesic on the integral manifold $N$ of the {\bf k}-nullity foliation in $G$. We shall prove that, $\gamma$ can be extended to a geodesic $\tilde{\gamma}:[0,\infty)\longrightarrow N$ on $N$. We shall proceed the proof with the contrary assumption, by supposing that such geodesic does not exist.
Following Theorem \ref{lema}, every ${\bf k}$-nullity manifold is auto-parallel and hence it is totally geodesic.
Therefore, $\gamma$ is a geodesic on $M$ and has an extension $\tilde{\gamma}:[0,\infty)\longrightarrow M$ such that $\gamma=\tilde{\gamma}\cap N$. It follows that, $p=\tilde{\gamma}(c)\notin G$. Suppose that $p_0=\gamma(0)=\tilde{\gamma}(0)$ and put $r_0=\mu_{\bf k}(p_0)$, the dimension of the {\bf k}-nullity space at $p_0$. The function $\mu_{\bf k}:M\longrightarrow M$ attains its minimum on $G$ and it results that $\mu_{\bf k}(p)>r_0$. Consider a basis  ${\cal B}_0=\{{\bf e}_1,{\bf e}_2,...,{\bf e}_{r_0},{\bf e}_{r_0+1},...,{\bf e}_n\}$ for $T_{p_{_0}}M$ such that $\{{\bf e}_1,{\bf e}_2,...,{\bf e}_{r_0}\}$ is a basis for ${\cal N}^{\bf k}_{p_{_0}}$ and ${\bf e}_1$ is the tangent vector to $\gamma$ at the point $p_0=\gamma(0)$. Using the following system of differential equations
\[
\frac{\nabla {\bf E}_i}{dt}=0,\ \ \ \ \ {\bf E}_i(0)={\bf e}_i,
\]
where $i=1,2,...,n$, one can translate the basis ${\cal B}_0$ into the parallel frame ${\cal B}=\{{\bf E}_1,{\bf E}_2,...,{\bf E}_{r_0},{\bf E}_{r_0+1},...,{\bf E}_n\}$ along $\tilde{\gamma}$. There is a neighborhood $U$ of $p$ on $M$ such the subset $\{{\bf E}_1,{\bf E}_2,...,{\bf E}_{r_0}\}$ is a basis for the ${\bf k}$-nullity space at every point $\tilde{\gamma}(t)$ in $G\cap U$. Since $\mu_{\bf k}(p)>r_0$, there is a vector field ${\bf E}_a$ along $\tilde{\gamma}$, for a fixed number $a$ in the range $r_0+1,...,n$, such that for every $t\in[0,c)$, we have ${\bf E}_a(t)\in\overset{\perp}{{\cal N}}{^{\bf k}_{\gamma(t)}}$ and ${\bf E}_a(c)\in{\cal N}^{\bf k}_{p}$. Now, let $\hat{\tilde{\gamma}}=(\tilde{\gamma},\dot{\tilde{\gamma}})$ be the natural lift of $\tilde{\gamma}$ to $TM_0$ and $\hat{{\cal B}}=\{\hat{{\bf E}}_1,\hat{{\bf E}}_2,...,\hat{{\bf E}}_{r_0},\hat{{\bf E}}_{r_0+1},...,\hat{{\bf E}}_n\}$ the basis for ${\cal H}_{\hat{\tilde{\gamma}}(t)}TM$ such that $\varrho(\hat{{\bf E}}_i)={\bf E}_i$. Assume that the coefficients $f_{ija}$ are defined as follows
\begin{equation}
\label{f}
\bar{\Omega}(\hat{{\bf E}}_i,\hat{{\bf E}}_j){\bf E}_a=f_{ija}.
\end{equation}
Using the relation (2) in Lemma \ref{lemi}, the Cartan horizontal derivative of both sides of Eq.(\ref{f}) along $\hat{\tilde{\gamma}}$ in $\pi^{-1}(G\cap U)$ and using the fact that, $\bar{\Omega}(\hat{{\bf E}}_1,V[\hat{{\bf E}}_j,\hat{{\bf E}}_i])=0$, we obtain
\begin{equation}
\label{Fin}
f^\prime_{ija}+\bar{\Omega}(\hat{{\bf E}}_j,[\hat{{\bf E}}_i,\hat{{\bf E}}_1])+\bar{\Omega}(\hat{{\bf E}}_i,[\hat{{\bf E}}_1,\hat{{\bf E}}_j])=0,
\end{equation}
where, $i,j=r_0+1,...,n$. Plugging $\hat{{\bf E}}_j$, $\hat{{\bf E}}_i$ and $\hat{{\bf E}}_1$ instead of $\hat{X}$, $\hat{Y}$ and $\hat{Z}$ into Eq.(\ref{K2}) and Eq.(\ref{K3}) respectively, we obtain
\[
\bar{\Omega}(\hat{{\bf E}}_j,V[\hat{{\bf E}}_i,\hat{{\bf E}}_1])+\bar{\Omega}(\hat{{\bf E}}_i,V[\hat{{\bf E}}_1,\hat{{\bf E}}_j])=0.
\]
Therefore, Eq.(\ref{Fin}) becomes
\begin{equation}
\label{FinFin}
f^\prime_{ija}+\bar{\Omega}(\hat{{\bf E}}_j,H[\hat{{\bf E}}_i,\hat{{\bf E}}_1])+\bar{\Omega}(\hat{{\bf E}}_i,H[\hat{{\bf E}}_1,\hat{{\bf E}}_j])=0.
\end{equation}
But, the horizontal part of $[\hat{{\bf E}}_j,\hat{{\bf E}}_1]$ can be written in the basis $\hat{{\cal B}}$ in the form
\[
H[\hat{{\bf E}}_1,\hat{{\bf E}}_j]={\bf W}^k_j\hat{{\bf E}}_k+{\bf W}^a_j\hat{{\bf E}}_a,
\]
for some functions ${\bf W}^k_j$ defined on $\hat{\tilde{\gamma}}$ in $\pi^{-1}(U)$, where the index $k$ runs over the range $1,...,\hat{a},...n$ and the hat over $a$ indicates that the index $a$ is omitted. Plugging the terms $H[\hat{{\bf E}}_j,\hat{{\bf E}}_1]$ and $H[\hat{{\bf E}}_1,\hat{{\bf E}}_i]$ into Eq.(\ref{FinFin}), we obtain the homogenous system of ODEs
\[
f^\prime_{ija}+{\bf W}^k_if_{jka}-{\bf W}^k_jf_{ika}=0.
\]
Since ${\bf E}_a$ is a ${\bf k}$-nullity vector field at $p$, by means of Eq.(\ref{f}), we have clearly  for the fixed index $a$, $f_{lma}(c)=0$, where, $l,m=r_0+1,...,n$. Solving the system of ODEs above with initial value $f_{lma}(c)=0$ implies that $f_{lma}\equiv0$. Eq.(\ref{f}), implies that, ${\bf E}_a$ is a ${\bf k}$-nullity vector at every point of $\tilde{\gamma}$ in $G\cap U$ and specially, it is a ${\bf k}$-nullity vector at every point of $\gamma$ in $G\cap U$. Obviously, this is merely a contradiction to the contrary hypothesis and $\gamma$ can be extended to a geodesic $\tilde{\gamma}:[0,\infty)\longrightarrow N$. \qed

 We remark that, relaxing the constant ${\bf k}$ to be zero in the Eq.(\ref{def bar omega}) leads to a notion of non-Riemannian nullity in Finsler geometry which is a special case of the nullity space in \cite{A5}.

\end{document}